\documentclass[11pt,reqno]{amsart}
  \usepackage{hyperref}
  \usepackage{amsmath,amssymb}
  \usepackage{xcolor}
  \usepackage{lineno}
  \usepackage{amsthm}
  \usepackage{centernot}
  \theoremstyle{plain}

  \newtheorem{thm}{Theorem}[section]
  
  \newtheorem{cor}[thm]{Corollary}

  \theoremstyle{definition}
  \newtheorem{defn}[thm]{Definition}
  \newtheorem{example}[thm]{Example}

  \numberwithin{equation}{section}

  \newcommand{\interior}[1]{{\kern0pt#1}^{\mathrm{o}}}

  \makeatletter
  \@namedef{subjclassname@2020}{%
  	\textup{2020} Mathematics Subject Classification}
  \makeatother

\begin{document}
  	
  	\title{Lebesgue Number and Total Boundedness}
  	\author[A. Gupta]{Ajit Kumar Gupta$^{\dagger *}$}
  	\address{$^\dagger$Department of Mathematics\\ National Institute of Technology Meghalaya\\ Shillong 793003\\ India}
  	\email{ajitkumar.gupta@nitm.ac.in}
  	\author[S. Mukherjee]{Saikat Mukherjee{$^{\dagger}$}}
  	\address{$^\dagger$Department of Mathematics\\ National Institute of Technology Meghalaya\\ Shillong 793003\\ India}
  	\email{saikat.mukherjee@nitm.ac.in}
  	$\thanks{*Corresponding author}$
  	

  	\begin{abstract}
  		A generalization of the Lebesgue number lemma is obtained. It is proved that, if each countably infinite locally finite open cover of a chainable metric space $X$ has a Lebesgue number, then $X$ is totally bounded. A property of metric spaces which is a generalization of connectedness and Menger convexity is introduced. It is observed that Atsujiness and compactness are equivalent for a metric space with this introduced property as well as for a chainable metric space.
  	\end{abstract}

  	\maketitle
  	
  	\section{Introduction}\label{intro}
  	The Lebesgue number lemma (\cite{jr11}, p. 175) says, if $X$ is a compact metric space, then each open cover of $X$ has a Lebesgue number.
	  A metric space $X$ is said to be \textit{Atsuji space} if each open cover of $X$ has a Lebesgue number. Atsuji spaces are weaker than compact metric spaces and stronger than complete metric spaces. In this manuscript, we obtain a generalization of the Lebesgue number lemma. 
%
%
	 Using this generalization, we immediately conclude that each locally finite open cover of a totally bounded subset $S$ in a complete metric space has a finite subcover for $S$. Further, we find a sufficient condition with which a chainable metric space is totally bounded, and as a corollary to it, we observe that Atsujiness and compactness are equivalent for a chainable metric space. A property of a metric space which is weaker than connectedness is introduced, and it is found that Atsujiness and compactness are equivalent for a metric space having this introduced property.
 
%
%
%

  	\section{Preliminaries}\label{prelim}
  	
  	For a subset $A$ in a metric space $X$, we denote the diameter of $A$ by diam$(A)$ and the complement of $A$ in $X$ by $X\setminus A$. An open ball centered at $a$ with radius $\epsilon$ is denoted by $B(a,\epsilon)$. 
  	
  	Let $(X,d)$ be a metric space. For a given $\epsilon>0$, an \textit{$\epsilon$-chain of length $n$} between two points $x,y\in X$ is a finite sequence $a_0,a_1,...,a_n$ in $X$ such that $a_0=x, a_n=y$ and $d(a_{i-1},a_i)\leq \epsilon$ for all $i=1,2,...,n$.
  A metric space $X$ is said to be \textit{finitely chainable} if for each $\epsilon>0$, there are finitely many points $p_1,p_2,...,p_j \in X$ and $m\in \mathbb N$ such that each point of $X$ can be joined with some $p_i$, $1\leq i \leq j$, by an $\epsilon$-chain of length $m$.  A subset $S$ in a metric space $X$ is said to be \textit{finitely chainable} if for each $\epsilon>0$ there exist finitely many points $p_1,p_2,...,p_j$ in $X$ and $m\in \mathbb N$ such that each point of $S$ can be joined with some $p_i$, $1\leq i\leq j$, by an $\epsilon$-chain of length $m$. By a finitely chainable subspace $S$ in a metric space $(X,d)$, we mean the metric space $(S,d)$ is finitely chainable. A totally bounded subset in a metric space $X$ is a finitely chainable subspace in $X$.

  	 A metric space is said to have an \textit{Atsuji completion} if its completion is an Atsuji space.
  	  	 
   	A topological space $X$ is said to be \textit{paracompact} if for each open cover of $X$ there is a locally finite open refinement that covers $X$.

%
  	 


  	\section{Lebesgue Number and Total Boundedness}\label{Lbsg No.-Ttly Bndd}	

Given an open cover $\mathcal O$ of a subset $S$ in a metric space $X$, if there is a number $\alpha >0$ such that each $A\subset S$ with diam$(A)<\alpha$ is contained in some $O\in \mathcal O$, then the number $\alpha$ is said to be a \textit{Lebesgue number} for the cover $\mathcal O$. 
 
Since, for a metric space $X$, each open cover of $X$ has a Lebesgue number if and only if each locally finite open cover of $X$ has a Lebesgue number, so the Lebesgue number lemma (\cite{jr11}, p. 175) can be equivalently restated as: Let $X$ be a complete metric space. If $X$ is totally bounded, then each locally finite open cover of $X$ has a Lebesgue number. However, each locally finite open cover of a totally bounded subset in a complete metric space need not have a Lebesgue number, for example, the totally bounded subset $(0,1)\cup(1,2)$ in $\mathbb R$. The following result provides a generalization of the Lebesgue number lemma.

\begin{thm}\label{TtlyBnddLclyFntCvr}
	Let $X$ be a complete metric space. If $S\subset X$ is totally bounded, then, for each locally finite open cover $\mathcal O$ of $S$ there is a number $\alpha >0$ such that each subset $B\subset S$ with diam$(B)<\alpha$ has a finite subcover from $\mathcal O$. 
\end{thm}
\begin{proof}
	 For a locally finite open cover $\mathcal O$ of a totally bounded subset $S\subset X$, consider the open cover $\mathcal A$ of $S$ consisting of all finite unions of elements of $\mathcal O$. We prove, there is a Lebesgue number $\alpha>0$ for the cover $\mathcal A$. On the contrary, we assume, there is no Lebesgue number for the cover $\mathcal A$. Then, for each $n$, there is $C_n \subset S$ with diam$(C_n)<1/n$ such that $C_n\not\subset A$ for all $A\in \mathcal A$. Consider an element $x_n\in C_n$. Since $S$ is totally bounded, the sequence $\{x_n\}$ has a Cauchy subsequence, say $\{x_{n_i}\}_{i=1}^\infty$, converging to some $a\in X$. If $a\in S'$, and since $\mathcal O$ is locally finite, there is an $\epsilon>0$ such that the ball $B(a,\epsilon)$ will intersect at most finitely many members of ${\mathcal O}$. For $\epsilon/2$, there is $i$ such that diam$(C_{n_i})<\epsilon/2$, which implies $C_{n_i}\subset B(x_{n_i},\epsilon/2)$, and for $\epsilon/2$, there is $j$ such that $x_{n_j}\in B(a,\epsilon/2)$. Taking $m=\max\{i,j\}$, we have $C_{n_m}\subset B(a,\epsilon)$, and since $C_{n_m}\subset S$, so $C_{n_m}$ is contained in a union of some finite number of elements of $\mathcal O$. On the other side, if $x_{n_i}=a$ for all $i$, then $a$ is in some $O\in \mathcal O$, which implies $B(a,\epsilon)\subset O$ for some $\epsilon>0$. For the $\epsilon>0$, there is $m$ such that diam$(C_{n_m})<1/{n_m}<\epsilon$, this implies $C_{n_m}\subset B(x_{n_m},\epsilon)=B(a,\epsilon)\subset O$. Thus, in either case, we reach a contradiction to the assumption.
\end{proof}


\begin{cor}
	Let $X$ be a complete, finitely chainable metric space. If a subset $S$ in $X$ has an Atsuji completion, then, for each locally finite open cover $\mathcal O$ of $S$ there is an $\alpha>0$ such that each subset $B\subset S$ with diam$(B)<\alpha$ has a finite subcover from $\mathcal O$.
\end{cor}

\begin{proof}
	Clearly $S$ is a finitely chainable subset. If $S$ is finitely chainable subspace, then by Theorem 3.5 in \cite{sk17}, it is totally bounded. Otherwise we proceed as follows. We use the notations used in the proof of above theorem. We assume, if possible, there is no Lebesgue number for the cover $\mathcal A$ of $S$. We get a sequence $\{x_n\}\subset S$ with $x_n\in C_n$. If the set $A=\{x_n:n\in \mathbb N\}$ is infinite, with no loss of generality we assume all $x_n$'s are distinct, then by Proposition 3.3 in \cite{sk17} we have $I(x_{n})$ converges to $0$, where $I(x_{n})=d(x_{n},X\setminus \{x_{n}\})$. Then by Theorem 3.8 in \cite{tj07}, there is a subsequence $\{x_{n_i}\}_{i=1}^\infty$ convergent to a point in $X$. Hence, if $A$ is infinite or finite, a contradiction to our assumption arises.
\end{proof}

Theorem \ref{TtlyBnddLclyFntCvr} may not hold if $X$ is not complete. For instance,
\begin{example}
	Consider the subspace $X=(0,2)\subset \mathbb R$, and a totally bounded subset $S=\{1/n:n\in \mathbb N\}$ in $X$. Let $r_n= 1/n-1/(n+1)$. Then, $\mathcal O=\{B(1/n,r_n):n\in \mathbb N\}$ is a locally finite open cover of $S$, and for each $\epsilon>0$ the subset $\{1/n:1/n<\epsilon\}\subset S$ has no finite subcover from $\mathcal O$.
\end{example}

The converse of Theorem \ref{TtlyBnddLclyFntCvr} is not true. For example,

\begin{example}
	Consider the Banach space $X=(l_1,\|\cdot\|_1)$, and a subset $S=\{e_m/n:m,n\in \mathbb N\}\cup \{(0,0,...)\}$ of it, where $e_m$ denotes the element having the $m$th element $1$ and all others $0$. Since $S$ is an Atsuji subset, so each covering of $S$ by sets open in $X$ has a Lebesgue number. But, the set $S$ is not totally bounded.
\end{example}

For a metric space $X$, the following are equivalent (Theorem 3.5, \cite{sk17}): 
\begin{enumerate}
	\item $X$ is totally bounded;
	\item Each locally finite open cover of $X$ with a Lebesgue number has a finite subcover.
\end{enumerate}
Being more general, for a subset $S$ in a metric space $X$ the following are equivalent too:
\begin{enumerate}
	\item $S$ is totally bounded;
	\item Each locally finite open cover of $S$ with a Lebesgue number has a finite subcover for $S$.
\end{enumerate}
So, in general, those locally finite open covers of a totally bounded subset which have the Lebesgue numbers have finite subcovers. A question arises: Can those locally finite open covers of a totally bounded subset which do not have any Lebesgue numbers have finite subcovers?

\begin{thm}\label{LclFntCvr-of-TtlBndd}
	Let $S$ be a subset in a complete metric space $X$. Then, $S$ is totally bounded if and only if each locally finite open cover of $S$ has a finite subcover for $S$.
\end{thm}
\begin{proof} 
	Let $\mathcal O$ be a locally finite open cover of $S$. By Theorem \ref{TtlyBnddLclyFntCvr}, there is an $\alpha >0$ such that each subset $B\subset S$ with diam$(B)<\alpha$ is covered by finitely many members of $\mathcal O$. Now, by the total boundedness of $S$, for the number $\alpha/3>0$, there are finitely many points $s_1,s_2,...,s_n\in S$ such that $S=\bigcup\limits_{i=1}^n[B(s_i,\alpha/3)\cap S]$. And thus, the cover $\mathcal O$ has a finite subcover for $S$.
	
	The converse part is proved by using the paracompactness of $X$.	
\end{proof}

\begin{cor}
	Let $S$ be a nonempty subset in an Atsuji space $X$. Then, $S$ is a finitely chainable subspace if and only if each locally finite open cover of $S$ has a finite subcover for $S$.
\end{cor}
\begin{proof}
	By using Theorem 3.5 in \cite{sk17}, we have: Given a subset $S$ in an Atsuji space $X$, $S$ is totally bounded if and only if $S$ is finitely chainable subspace. 
\end{proof}

Due to the paracompactness of a metric space, we get the following.
\begin{cor}
	If $S$ is a totally bounded subset in a complete metric space $X$, then each open cover of $X$ has a finite subcover for $S$. 
\end{cor}

We introduce now a property of a metric space $X$ which generalizes the connectedness as well as the Menger convexity of $X$. We will study a relation between Lebesgue number and total boundedness in the spaces which are more general than connected spaces. We look at some definitions and existing results first as follows.

 A metric space $X$ is said to be \textit{chainable} if for each $\epsilon>0$ each pair of points $x,y$ can be joined by an $\epsilon$-chain. All connected metric spaces are chainable. 

A  metric space $X$ is said to be\textit{ metrically convex} if for any two distinct points $x,y\in X$ there is a $z\in X\setminus \{x,y\}$ such that $d(x,y)=d(x,z)+d(z,y)$. A metric space $X$ is said to be \textit{Menger convex} (\cite{ib05}) if for each $x,y\in X$ and each $0\leq r \leq d(x,y)$, the set $B[x,r]\cap B[y,d(x,y)-r]\neq \emptyset $. Clearly, a Menger convex metric space is always metrically convex; but the converse need not be true, for eaxmple the set of rationals $\mathbb Q$ is metrically convex but not Menger convex. 

\begin{defn}
	Let $(X,d)$ be a metric space. We say, it has the \textit{property $P$} if, for a given $x\in X$ and for all $r$ with $0\leq r< \sup\limits_{y\in X}d(x,y)$, there is a $z\in X$ such that $d(x,z)=r$.
\end{defn}

Let the class of metric spaces with the property $P$ be denoted by $\mathcal P$.  Then, $\mathcal P$ contains:
\begin{enumerate}
	\item All connected metric spaces: Given a pair of distinct points $x,y$ in a connected metric space $X$ and $0< r < d(x,y)$, the set $\{z\in X:d(x,z)=r\}$ can't be empty, otherwise the open ball $B(x,r)$ would be closed too. Singleton spaces obviously have the property $P$.
	
	\item Complete metrically convex spaces, because of Menger's Theorem \cite{am04}.
	
	\item Menger convex spaces: Given a pair of points $x,y$ in a metric space $X$ and $0\leq r\leq d(x,y)$, $B[x,r]\cap B[y,d(x,y)-r]=S[x,r]\cap S[y,d(x,y)-r]$, where $S[x,r]=\{z\in X:d(x,z)=r\}$.
	
	A Menger convex metric space need not be connected, for example, the subset $S= \{(x,y)\in \mathbb R^2:x^2+y^2\neq 1\}$ of $\mathbb R^2$ endowed with the metric $d((x_1,y_1),(x_2,y_2))=\max \{|x_1-x_2|,|y_1-y_2|\}$ is Menger convex space but not connected.
	
\end{enumerate}

 We note that a subspace in $\mathbb R^2$ defined as the union of two distinct parallel lines has the property $P$ but it is neither chainable nor metrically convex. The subspace $\mathbb Q$ of $ \mathbb R$ is metrically convex and chainable but it does not have the property $P$.
 
  M. H. A. Newman, in \cite{mh64} (Theorem 5.1, p. 81), proved that if each open cover of a chainable metric space has a finite subcover, then it is connected. By considering more general hypotheses, in the following result, we show that the space is totally bounded.
  
 \begin{thm}\label{ChnblLbsgTtlybndd}
	Let $X$ be a chainable metric space. If every countably infinite locally finite open cover of $X$ has a Lebesgue number, then $X$ is totally bounded.
\end{thm}

\begin{proof}
	If possible, suppose $(X,d)$ is not totally bounded. Then, there is a sequence $\{x_n\}$ for which there is some $\epsilon>0$ such that $d(x_m,x_n)>\epsilon$ for all $m\neq n$. Consider an open cover $\mathcal O=\{B(x_n,\epsilon/4n):n\in \mathbb N\}\cup \{X\setminus S\}$ of $X$, where $S=\{x_n:n\in \mathbb N\}$. For each $x\in X$, the ball $B(x,\delta)$ intersects at most two members of $\mathcal O$, where $0<\delta<\epsilon/4$. Hence, $\mathcal O$ is locally finite.
	
	Now, we prove $\mathcal O$ has no Lebesgue number. Let $0<\gamma<\epsilon$. Then, there is $s$ with $0<s<\gamma$, and there is some $p\in \mathbb N$ such that $\epsilon/4p<s$. Let $t= \gamma-s$. Then, since $X$ is chainable, there are $a_0=x_p,a_1,...,a_r=x_m$ in $X$, $p\neq m$, such that $d(a_{i-1},a_i)\leq t$ for all $1\leq i\leq r$. This implies, there is some $i$ with $1\leq i\leq r-1$ such that $s\leq d(x_p,a_i)<\gamma$. And thus, there is no element in $\mathcal O$ which can contain the open ball $B(x_p,\gamma)$. Hence, the proof.
\end{proof}

 One of the main results in \cite{gb81} by G. Beer is: Given a chainable space $X$, $X$ is compact if and only if it is uniformly locally compact and uniformly chainable. Using this result we get,
 \begin{cor}\label{ChnbldAtsjCmpct}
 	For a chainable metric space $X$, the following are equivalent:
 	\begin{enumerate}
 		\item $X$ is Atsuji;
 		\item $X$ is compact;
 		\item $X$ is uniformly locally compact and uniformly chainable.
 	\end{enumerate}
 \end{cor}

 By using the technique used in the proof of Theorem \ref{ChnblLbsgTtlybndd}, we get the following.
  
\begin{thm}\label{CnnctdLbsgTtlybndd}
	Let $X$ be a metric space with the property $P$. If every countably infinite locally finite open cover of $X$ has a Lebesgue number, then $X$ is totally bounded.
\end{thm}

\begin{cor}\label{PropP-AtsjIsCmpct}   
	Let $X$ be a metric space with the property $P$. Then, $X$ is Atsuji if and only if it is compact.
\end{cor}
 The converses of Theorems \ref{ChnblLbsgTtlybndd}, \ref{CnnctdLbsgTtlybndd} are not true. For instance:
\begin{example}\label{(0,1].}
	Consider the subspace $X=(0,1]\subset \mathbb R$ which is connected. The space $X$ is totally bounded. Let $r_n=1/n-1/(n+1), n\in \mathbb N$. Then, the locally finite open cover $\mathcal O=\{B(1/n,r_n)\subset X:n\in \mathbb N\}$ of $X$ has no Lebesgue number.
\end{example}

  \end{document}